\RequirePackage{fix-cm}
\documentclass[smallextended]{svjour3}       
\smartqed  
\usepackage{graphicx}

 \usepackage{mathptmx}      
\usepackage{latexsym}
\begin{document}

\title{Conformal Submersion with Horizontal Distribution
}


\author{Mahesh T V and K S Subrahamanian Moosath$^{*}$\thanks{ *corresponding author}}


\institute{Department of Mathematics \at
              Indian Institute of Space Science and Technology \\
              Thiruvanathapuram, Kerala, India-695547\\
              \email{maheshtv.16@res.iist.ac.in(Mahesh T V)} \\ 
              \email{smoosath@iist.ac.in (K S Subrahamanian Moosath)}         
           }

\date{Received: date / Accepted: date}
\maketitle

\begin{abstract}
In this article, conformal submersion with horizontal distribution of Riemannian manifolds is defined which is a generalization of the affine submersion with horizontal distribution. Then, a necessary condition is obtained for the existence of a conformal submersion with horizontal distribution. For the dual connections $\nabla$ and $\overline{\nabla}$ on manifold $\mathbf{M}$ and $\nabla^*$ and $\overline{\nabla}^*$ on manifold $\mathbf{B}$, we show that $\pi: (\mathbf{M},\nabla) \longrightarrow (\mathbf{B}, \nabla^{*}) $ is a conformal submersion with horizontal distribution if and only if  $\pi: (\mathbf{M},\overline{\nabla}) \longrightarrow (\mathbf{B}, \overline{\nabla^{*}}) $ is a conformal submersion with  horizontal distribution. Also, we obtained a necessary and sufficient condition for $\pi \circ \sigma$ to become a geodesic of $\mathbf{B}$ if $\sigma$ is a geodesic of $\mathbf{M}$ for $ \pi: (\mathbf{M},\nabla,g_{m}) \rightarrow (\mathbf{B},\nabla^{*},g_{b})$ a conformal submersion with horizontal distribution.
\end{abstract}

\section{Introduction}
Riemannian submersion is a special tool in differential geometry and it has got application in different areas such as Kaluza-Klein theory, Yang-Mills theory, supergravity and superstring theories, statistical machine learning processes, medical imaging, theory of robotics and the statistical analysis on manifolds. O'Neill \cite{o1966fundamental}  defined a  Riemannian submersion and obtained the fundamental equations of Riemannian submersion for Riemannian manifolds. In \cite{o1967submersions}, O'Neill compare the geodesics of $\mathbf{M}$ and $\mathbf{B}$ for a semi-Riemannian submersion $\pi: \mathbf{M}\longrightarrow \mathbf{B}$. Abe and Hasegawa \cite{abe2001affine} defined an affine submersion with horizontal distribution which is a dual notion of affine immersion and obtained the fundamental equations. They compare the geodesics of $\mathbf{M}$ and $\mathbf{B}$ for an affine submersion with horizontal distribution $\pi: \mathbf{M}\longrightarrow \mathbf{B}$. 

Conformal submersion and the fundamental equations of conformal submersion were also studied by many researchers, see \cite{lornia1993fe}, \cite{gud1990har} for example. Horizontally conformal submersion is a generalization of the Riemannian submersion. Horizontally conformal submersion is a special horizontally conformal map which got introduced independently by Fuglede \cite{fuglede1978harmonic}  and Ishihara \cite{ishihara1978mapping}. Their study  focuses on the conformality relation between metrics on the Riemannian manifolds and  Levi-Civita connections. Here the study is on the conformal submersion between Riemannian manifolds $\mathbf{M}$ and  $\mathbf{B}$ and the conformality relation between any two affine connections $\nabla$ and $\nabla^{*}$ (need not necessarily be the Levi-Civita connection ) on $\mathbf{M}$ and  $\mathbf{B}$, respectively. 

In section 2, relevant basic concepts are given. In section 3, the concept of a conformal submersion with horizontal distribution is defined and a necessary condition is obtained for the existence of such maps. Fundamental equations of conformal submersion with horizontal distribution is given and show that $\pi: (\mathbf{M},\nabla) \longrightarrow (\mathbf{B}, \nabla^{*}) $ is a conformal submersion with horizontal distribution if and only if  $\pi: (\mathbf{M},\overline{\nabla}) \longrightarrow (\mathbf{B}, \overline{\nabla^{*}}) $ is a conformal submersion with horizontal distribution. In section 4, we obtained a necessary and sufficient condition for $\pi \circ \sigma$ to be a geodesic of $\mathbf{B}$ when $\sigma$ is a geodesic of $\mathbf{M}$ for $ \pi: (\mathbf{M},\nabla,g_{m}) \rightarrow (\mathbf{B},\nabla^{*},g_{b})$ a conformal submersion with horizontal distribution. Throughout this paper, all the objects are assumed to be smooth.
\section{Affine submersion with horizontal distribution}
In this section, the concepts like submersion with horizontal distribution and affine submersion with horizontal distribution are given.

Let $\mathbf{M}$ and $\mathbf{B}$ be Riemannian manifolds with dimension $n$ and $m$ respectively with $n>m$. An onto map $\pi :\mathbf{M}\longrightarrow \mathbf{B}$ is called a submersion if $\pi_{*p} : T_{p}\mathbf{M} \longrightarrow T_{\pi(p)}\mathbf{B}$ is onto for all $p \in \mathbf{M}$. For a submersion $\pi :\mathbf{M}\longrightarrow \mathbf{B}$, $\pi^{-1}(b)$ is a submanifold of $\mathbf{M}$ of dimension $(n-m)$ for each $b \in \mathbf{B} $. These submanifolds $\pi^{-1}(b)$  are called fibers. Set $\mathcal{V}(\mathbf{M})_{p}$ = $Ker(\pi_{*p})$ for each $p \in \mathbf{M}$.

\begin{definition}
A submersion $\pi :\mathbf{M}\longrightarrow \mathbf{B}$ is called a submersion with horizontal distribution if there is a smooth distribution $p\longrightarrow \mathcal{H}(M)_{p}$ such that 
\begin{equation}
T_{p}\mathbf{M} = \mathcal{V}(\mathbf{M})_{p} \bigoplus \mathcal{H}(M)_{p}.
\end{equation}
\end{definition}

We call $\mathcal{V}(\mathbf{M})_{p}$ ($\mathcal{H}(M)_{p}$) the vertical (horizontal) subspace of $T_{p}\mathbf{M}$. $ \mathcal{H}$ and $\mathcal{V}$ denote the projections of the tangent space of $\mathbf{M}$ onto the horizontal and vertical subspaces, respectively.
\begin{note}
Let $\pi :\mathbf{M}\longrightarrow \mathbf{B}$  be a submersion with horizontal distribution  $\mathcal{H}(M)$. Then, $\pi_{*}\mid_{\mathcal{H}(M)_{p}}:\mathcal{H}(M)_{p} \longrightarrow T_{\pi(p)}\mathbf{B}$ is an isomorphism for each $p \in \mathbf{M}$.
\end{note}

\begin{definition}
A vector field $Y$ on $\mathbf{M}$ is said to be projectable if there exists a vector field $Y_{*}$ on $\mathbf{B}$ such that $\pi_{*}(Y_{p})$ = $Y_{*\pi(p)}$ for each $p \in \mathbf{M}$, that is $Y$ and $Y_{*}$ are $\pi$- related. A vector field $X$ on $\mathbf{M}$ is said to be basic if it is projectable and horizontal. Every vector field $X$ on $\mathbf{B}$ has a unique smooth horizontal lift,  denoted by $\tilde{X}$, to $\mathbf{M}$. 
\end{definition}

\begin{definition}
Let $\nabla$ and $\nabla^{*}$ be affine connections on $\mathbf{M}$ and $\mathbf{B}$ respectively. $\pi : (\mathbf{M},\nabla) \longrightarrow (\mathbf{B},\nabla^{*})$ is said to be an affine submersion with horizontal distribution if $\pi: \mathbf{M}\longrightarrow \mathbf{B}$ is a submersion with horizontal distribution and satisfies $\mathcal{H}(\nabla_{\tilde{X}} \tilde{Y})$ = $(\nabla^{*}_{X}Y)^{\tilde{}}$, for vector fields $X,Y$ in $\mathcal{X}(\mathbf{B})$, where $\mathcal{X}(\mathbf{B})$ denotes the set of all vector fields on $\mathbf{B}$.
\end{definition}

\begin{note} 
Abe and Hasegawa \cite{abe2001affine} proved that the connection $\nabla$ on $\mathbf{M}$ induces a connection $\nabla^{'}$ on $\mathbf{B}$ when $\pi: \mathbf{M}\longrightarrow \mathbf{B}$ is a submersion with horizontal distribution and $\mathcal{H}(\nabla_{\tilde{X}}\tilde{Y})$ is projectable for all vector fields $X$ and $Y$ on $\mathbf{B}$. 

A connection $\mathcal{V}\nabla \mathcal{V}$ on the subbundle $\mathcal{V}(\mathbf{M})$ is defined by $(\mathcal{V}\nabla \mathcal{V})_{E}V$ = $\mathcal{V}(\nabla_{E}V)$ for any vertical vector field $V$ and any vector field $E$ on $\mathbf{M}$. For each $b \in \mathbf{B}$,  $\mathcal{V}\nabla \mathcal{V}$ induces a unique connection $\hat{\nabla}^{ b}$ on the fiber $\pi^{-1}(b)$. Abe and Hasegawa \cite{abe2001affine} proved that if $\nabla$ is torsion free, then $\hat{\nabla}^{b}$ and $\nabla^{'}$ are also torsion free.
\end{note}


\begin{definition}
 Let $\pi: (\mathbf{M},\nabla) \longrightarrow (\mathbf{B},\nabla^{*})$ be an affine submersion with horizontal distribution $\mathcal{H}(M)$. Then, the fundamental tensors $T$ and $A$ are defined as
\begin{eqnarray}
T_{E}F &=& \mathcal{H}(\nabla_{\mathcal{V}E}\mathcal{V}F)+\mathcal{V}(\nabla_{\mathcal{V}E}\mathcal{H}F),\\
A_{E}F &=&  \mathcal{V}(\nabla_{\mathcal{H}E}\mathcal{H}F)+\mathcal{H}(\nabla_{\mathcal{H}E}\mathcal{V}F),
\end{eqnarray}
for $E$ and $F$ in $\mathcal{X}(\mathbf{M})$.
\end{definition}

\begin{remark}
For a Riemannian manifold $(\mathbf{M},g)$ with an affine connection $\nabla$ one can define the dual connection  $\overline{\nabla}$ by  $ Xg(Y,Z) = g(\nabla_{X}Y,Z) + g(Y,\overline{\nabla}_{X} Z)$, for   $X,Y$ and $Z$ in $\mathcal{X}(\mathbf{M})$. The fundamental tensors correspond to the dual connection $\overline{\nabla}$ is denoted by $\overline{T}$ and $\overline{A}$.
\end{remark}

Note that these are $(1, 2)$-tensors and these tensors can be defined in a general situation, namely, it is enough that a manifold $\mathbf{M}$ has a splitting $T\mathbf{M} = \mathcal{V}(\mathbf{M}) \bigoplus \mathcal{H}(\mathbf{M})$. Also, note that $T_{E}$ and $A_{E}$ reverses the horizontal and vertical subspaces and $T_{E} = T_{\mathcal{V}E}$, $A_{E} = A_{\mathcal{H}E}$.
\section{Conformal submersion with horizontal distribution}

In this section, we generalize the concept of affine submersion with horizontal distribution to conformal submersion with horizontal distribution and prove a necessary condition for the existence of such maps. Then, we show that $\pi: (\mathbf{M},\nabla) \longrightarrow (\mathbf{B}, \nabla^{*}) $ is a conformal submersion with horizontal distribution if and only if  $\pi: (\mathbf{M},\overline{\nabla}) \longrightarrow (\mathbf{B}, \overline{\nabla^{*}}) $ is a conformal submersion with horizontal distribution. 

\begin{definition}
Let $(\mathbf{M},g_{m})$ and $(\mathbf{B},g_{b})$ be Riemannian manifolds. A submersion $\pi: (\mathbf{M},g_{m}) \longrightarrow (\mathbf{B},g_{b}) $ is called a conformal submersion if there exists a $\phi \in C^{\infty}(\mathbf{M})$ such that 
\begin{equation}
g_{m}(X,Y) = e^{2 \phi}g_{b}(\pi_{*}X,\pi_{*}Y),
\end{equation}
for all horizontal vector fields $X,Y \in \mathcal{X}(\mathbf{M})$.
\end{definition}

For $\pi : (\mathbf{M},\nabla) \longrightarrow (\mathbf{B},\nabla^*) $ an affine submersion with horizontal distribution, $\mathcal{H}(\nabla_{\tilde{X}}\tilde{Y}) =(\nabla^{*}_{X}Y)^{\tilde{}}$, for $X,Y \in \mathcal{X}(\mathbf{B})$. In the case of conformal submersion we prove the following theorem, which is the motivation for us to generalize the concept of affine submersion with horizontal distribution. 

\begin{theorem}\label{cs12}
Let $\pi: (\mathbf{M},g_{m}) \longrightarrow (\mathbf{B},g_{b}) $  be a conformal submersion. If $\nabla$ on $\mathbf{M}$ and $\nabla^{*}$ on $B$ are the Levi-Civita connections, then 
\begin{eqnarray*}
\mathcal{H}(\nabla_{\tilde{X}}\tilde{Y}) &=& \tilde{(\nabla^{*}_{X}Y)} +\tilde{X}(\phi)\tilde{Y}+ \tilde{Y}(\phi)\tilde{X}-\mathcal{H}(grad_{\pi}\phi)g_{m}(\tilde{X},\tilde{Y}),
\end{eqnarray*}
where $X,Y\in \mathcal{X}(\mathbf{B})$  and $\tilde{X},\tilde{Y}$ denote their unique horizontal lifts on $\mathbf{M}$ and $grad_{\pi}\phi$ denote the gradient of $\phi$ with respect to $g_{m}$
\end{theorem}

\begin{proof}
We have the Koszul formula for the Levi-Civita connection,
\begin{small} 
\begin{eqnarray}\label{cs3}
2g_{m}(\nabla_{\tilde{X}}\tilde{Y},\tilde{Z}) &=& \tilde{X}g_{m}(\tilde{Y},\tilde{Z})+\tilde{Y}g_{m}(\tilde{Z},\tilde{X})-\tilde{Z}g_{m}(\tilde{X},\tilde{Y})-g_{m}(\tilde{X},[\tilde{Y},\tilde{Z}]) \nonumber \\ &&+g_{m}(\tilde{Y},[\tilde{Z},\tilde{X}])+g_{m}(\tilde{Z},[\tilde{X},\tilde{Y}]).
\end{eqnarray}
\end{small}
Now, consider 
\begin{eqnarray*}
\tilde{X}g_{m}(\tilde{Y},\tilde{Z}) &=& \tilde{X}(e^{2\phi}g_{b}(Y,Z))\\
                                    &=& \tilde{X}(e^{2\phi})g_{b}(Y,Z)+                          e^{2\phi}\tilde{X}(g_{b}(Y,Z))\\
                                    &=& 2e^{2\phi}\tilde{X}(\phi)g_{b}(Y,Z)+e^{2\phi}Xg_{b}(Y,Z).\\                                    
\end{eqnarray*}
Similarly, 
\begin{eqnarray*}
\tilde{Y}g_{m}(\tilde{X},\tilde{Z}) &=& 2e^{2\phi}\tilde{Y}(\phi)g_{b}(X,Z)+e^{2\phi}Yg_{b}(X,Z).\\
\tilde{Z}g_{m}(\tilde{X},\tilde{Y}) &=& 2e^{2\phi}\tilde{Z}(\phi)g_{b}(X,Y)+e^{2\phi}Zg_{b}(X,Y).
\end{eqnarray*}
Also, 
\begin{eqnarray*}
g_{m}(\tilde{X}, [\tilde{Y},\tilde{Z}]) &=& e^{2\phi}g_{b}(X,[Y,Z]).\\
g_{m}(\tilde{Y}, [\tilde{Z},\tilde{X}]) &=& e^{2\phi}g_{b}(Y,[Z,X]).\\
g_{m}(\tilde{Z}, [\tilde{X},\tilde{Y}]) &=& e^{2\phi}g_{b}(Z,[X,Y]).
\end{eqnarray*}
Then, from the equation (\ref{cs3}) and the above equations  
\begin{small}
\begin{eqnarray*}
2g_{m}(\nabla_{\tilde{X}}\tilde{Y},\tilde{Z})&=&2\tilde{X}(\phi)e^{2\phi}g_{b}(Y,Z)+ 2\tilde{Y}(\phi)e^{2\phi}g_{b}(X,Z)- 2\tilde{Z}(\phi)e^{2\phi}g_{b}(X,Y)\\&&+ 2e^{2\phi}g_{b}(\nabla^{*}_{X}Y,Z).
\end{eqnarray*}
\end{small}
Since, $\tilde{Z}(\phi)$ = $e^{2\phi}g_{b}(\pi_{*}(grad_{\pi}\phi), Z)$ we get
\begin{eqnarray*}
\pi_{*}(\nabla_{\tilde{X}}\tilde{Y}) &=& \nabla^{*}_{X}Y +\tilde{X}(\phi)Y+ \tilde{Y}(\phi)X - e^{2\phi}\pi_{*}(grad_{\pi}\phi)g_{b}(X,Y).
\end{eqnarray*}
Hence,  
\begin{eqnarray*}
\mathcal{H}(\nabla_{\tilde{X}}\tilde{Y}) &=& \tilde{(\nabla^{*}_{X}Y)} +\tilde{X}(\phi)\tilde{Y}+ \tilde{Y}(\phi)\tilde{X} -\mathcal{H}(grad_{\pi}\phi)g_{m}(\tilde{X},\tilde{Y}).
\end{eqnarray*}
\end{proof}

Now, we generalize the concept of affine submersion with horizontal distribution. 

\begin{definition} Let $\pi: (\mathbf{M},g_{m}) \longrightarrow (\mathbf{B},g_{b}) $  be a conformal submersion and let $\nabla$ and $\nabla^*$ be affine connections on $\mathbf{M}$ and $\mathbf{B}$, respectively. Then, $\pi : (\mathbf{M},\nabla) \longrightarrow (\mathbf{B}, \nabla^{*})$ is said to be a conformal submersion with horizontal distribution $\mathcal{H}(\mathbf{M}) = \mathcal{V}(\mathbf{M})^{\perp}$ if
 \begin{eqnarray*}
\mathcal{H}(\nabla_{\tilde{X}}\tilde{Y}) &=& \tilde{(\nabla^{*}_{X}Y)} +\tilde{X}(\phi)\tilde{Y}+ \tilde{Y}(\phi)\tilde{X} -\mathcal{H}(grad_{\pi}\phi)g_{m}(\tilde{X},\tilde{Y}),
\end{eqnarray*}
for some $\phi \in C^{\infty}(\mathbf{M})$ and for all $X,Y \in \mathcal{X}(\mathbf{B})$.
\end{definition}

\begin{note}
If $\phi$ is constant, it turns out to be an affine submersion with horizontal distribution. 
\end{note}

\begin{example}
Let $H^{n} = \lbrace (x_{1},...,x_{n}) \in \mathbf{R}^{n} : x_{n} > 0 \rbrace$ and $\tilde{g} = \frac{1}{x_{n}^{2}} g$ be a Riemannian metric on $H^{n}$, where $g$ is the Euclidean metric on $\mathbf{R}^{n}$. Let $\pi : H^{n} \longrightarrow \mathbf{R}^{n-1}$ be  defined by $\pi(x_{1},...,x_{n}) = (x_{1},...,x_{n-1}).$
 Let $\phi : H^{n} \longrightarrow \mathbf{R}$ be  defined by $\phi(x_{1},...,x_{n}) = \log(\frac{1}{x_{n}^{2}})$. Then, we have 
\begin{eqnarray*}
\tilde{g}\left(\frac{\partial}{\partial x_{i}}, \frac{\partial}{\partial x_{j}}\right) = e^{\phi} g\left(\frac{\partial}{\partial x_{i}}, \frac{\partial}{\partial x_{j}}\right).
\end{eqnarray*}
Hence, $\pi :(H^{n},\tilde{g}) \longrightarrow (\mathbf{R}^{n-1},g)$ is 
a conformal submersion. Then, by the theorem (\ref{cs12}), $\pi :(H^{n},\nabla) \longrightarrow (\mathbf{R}^{n-1},\nabla^{*})$ is a conformal submersion with horizontal distribution, where $\nabla$ and $\nabla^{*}$ are Levi-Civita connections on $H^{n}$ and $\mathbf{R}^{n-1}$, respectively.
\end{example}

Now, we prove a necessary condition for the existence of a conformal submersion with horizontal distribution.
\begin{theorem}
Let $\pi: (\mathbf{M},g_{m}) \longrightarrow (\mathbf{B},g_{b}) $  be a conformal submersion and $\nabla$ be an affine connection on $\mathbf{M}$. Assume that $\pi: \mathbf{M}\longrightarrow \mathbf{B}$ is a submersion with horizontal distribution. If $\mathcal{H}(\nabla_{\tilde{X}}\tilde{Y})$ is projectable for all vector fields $X$ and $Y$ on $\mathbf{B}$. Then, there exists a unique connection $\nabla^{*}$ on $\mathbf{B}$ such that $\pi : (\mathbf{M},\nabla) \longrightarrow (\mathbf{B}, \nabla^{*})$ is a conformal submersion with horizontal distribution.
\end{theorem}
\begin{proof}
Setting $\nabla^{*}_{X}Y$ = $\pi_{*}(\nabla_{\tilde{X}}\tilde{Y})-\tilde{X}(\phi)Y- \tilde{Y}(\phi)X + e^{2\phi}\pi_{*}(grad_{\pi}\phi)g_{b}(X,Y)$, we show that $\nabla^{*}$ is an affine connection on $\mathbf{B}$. For example, since $\tilde{(fX)} = (f \circ \pi)\tilde{X}$, we have
\begin{eqnarray}\label{csg1}
\nabla^{*}_{X}(fY) &=& \pi_{*}(\nabla_{\tilde{X}}(f\circ \pi)\tilde{Y})-\tilde{X}(\phi)fY- \tilde{(fY)}(\phi)X \nonumber\\&& + e^{2\phi}\pi_{*}(grad_{\pi}\phi)g_{b}(X,fY).
\end{eqnarray}
Now, consider
\begin{eqnarray}\label{csg2}
\pi_{*}(\nabla_{\tilde{X}}(f\circ \pi)\tilde{Y}) &=& \pi_{*}((\tilde{X}(f \circ \pi))\tilde{Y} + (f \circ \pi)\nabla_{\tilde{X}}\tilde{Y})\nonumber\\ &=& X(f)Y +f\nabla^{*}_{X}Y+\tilde{X}(\phi)fY+f\tilde{Y}(\phi)X \nonumber\\&& - e^{2\phi}\pi_{*}(grad_{\pi}\phi)fg_{b}(X,Y).
\end{eqnarray}
Then, from the equations (\ref{csg1}) and (\ref{csg2})  
\begin{eqnarray*}
\nabla^{*}_{X}(fY) &=&  X(f)Y +f\nabla^{*}_{X}Y.
\end{eqnarray*}
The other condition for affine connection can be proved similarly. The uniqueness is clear from the definition.
\end{proof}
 As in the case of affine submersion with horizontal distribution we have, 
 \begin{lemma} 
 Let $\pi : (\mathbf{M},\nabla) \longrightarrow (\mathbf{B}, \nabla^{*})$  be conformal submersion with horizontal distribution, then 
 \begin{eqnarray}
 \mathcal{H}(Tor(\nabla)(\tilde{X},\tilde{Y})) &=&\widetilde{(Tor(\nabla^{*})(X,Y))}.\\
  \mathcal{V}(Tor(\nabla)(V,W)) &=&(Tor(\hat{\nabla})(V,W)).
 \end{eqnarray}
 \end{lemma}
\begin{proof}
Proof follows immediately form the definition of the  conformal submersion with horizontal distribution.
\end{proof}
\begin{corollary}
If $\nabla$ is torsion-free, then $\nabla^{*}$ and $\hat{\nabla}$ are also torsion-free.
\end{corollary}

Fundamental tensors $T$ and $A$ for a conformal submersion with horizontal distribution $\pi : (\mathbf{M},\nabla) \longrightarrow (\mathbf{B}, \nabla^{*})$ are defined for $E$ and $F$ in $\mathcal{X}(\mathbf{M})$ by 
\[
T_{E}F = \mathcal{H}\nabla_{\mathcal{V}E}(\mathcal{V}F)+ \mathcal{V}\nabla_{\mathcal{V}E}(\mathcal{H}F).
\]
and
\[
A_{E}F = \mathcal{V}\nabla_{\mathcal{H}E}(\mathcal{H}F)+ \mathcal{H}\nabla_{\mathcal{H}E}(\mathcal{V}F).
\]
Note that these are $(1,2)$-tensors. \par 
 We have the fundamental equations correspond to the conformal submersion with horizontal distribution. Let $R$ be the curvature tensor of $(\mathbf{M},\nabla)$ defined by 
 \begin{eqnarray*}
 R(E,F)G = \nabla_{[X,Y]}G -\nabla_{E}\nabla_{F}G+ \nabla_{F}\nabla_{E}G,
 \end{eqnarray*}
for $E,F$ and $G$ in $\mathcal{X}(\mathbf{M})$. Similarly, we denote the curvature tensor of $\nabla^{*}$ (respectively $\hat{\nabla}$) by $R^{*}$ (respectively $\hat{R}$). Define $(1,3)$-tensors $R^{P_{1},P_{2},P_{3}}$ for conformal submersion with horizontal distribution by 
\begin{eqnarray*}
R^{P_{1},P_{2},P_{3}}(E,F)G &=& P_{3}\nabla_{[P_{1}E,P_{2}F]}P_{3}G -P_{3}\nabla_{P_{1}E}(P_{3}\nabla_{P_{2}F}P_{3}G)\\
&& + P_{3}\nabla_{P_{2}F}(P_{3}\nabla_{{P_{1}E}}P_{3}G),
\end{eqnarray*}
where $P_{i}= \mathcal{H}$ or $\mathcal{V}$($i=1,2,3$)and $E,F,G$ are in $\mathcal{X}(\mathbf{M})$. Then, the following fundamental equations for conformal submersion with horizontal distribution can be obtained. 
\begin{theorem}
 Let $X,Y,Z$ be horizontal and $U,V,W$ vertical vector fields in $\mathbf{M}$. Then,
  \begin{eqnarray*}
 \mathcal{V}R(U,V)W &=& R^{\mathcal{V}\mathcal{V}\mathcal{V}}(U,V)W+T_{V}T_{U}W-T_{U}T_{V}W.\\
  \mathcal{H}R(U,V)W &=& \mathcal{H}(\nabla_{V}T)_{U}W - \mathcal{H}(\nabla_{U}T)_{V}W-T_{Tor(\nabla)(U,V)}W.\\
\mathcal{V}R(U,V)X &=&  \mathcal{H}(\nabla_{V}T)_{U}X - \mathcal{V}(\nabla_{U}T)_{V}X -T_{Tor(\nabla)(U,V)}X.\\
  \mathcal{H}R(U,V)X &=& R^{\mathcal{V}\mathcal{V}\mathcal{H}}(U,V)X+T_{V}T_{U}X-T_{U}T_{V}X.\\
 \mathcal{V}R(U,X)V &=& R^{\mathcal{V}\mathcal{H}\mathcal{V}}(U,X)V-T_{U}A_{X}V-A_{X}T_{U}V.\\
   \mathcal{H}R(U,X)V &=& \mathcal{H}(\nabla_{X}T)_{U}V -\mathcal{H}(\nabla_{U}A)_{X}V-A_{A_{X}U}V+T_{T_{U}X}V\\&&-T_{Tor(\nabla)(U,X)}V-A_{Tor(\nabla)(U,X)}V.\\
   \mathcal{V}R(U,X)Y &=& \mathcal{V}(\nabla_{X}T)_{U}Y -\mathcal{V}(\nabla_{U}A)_{X}Y-A_{A_{X}U}Y+T_{T_{U}X}Y\\&&-T_{Tor(\nabla)(U,X)}Y-A_{Tor(\nabla)(U,X)}Y.\\
   \mathcal{H}R(U,X)Y &=& R^{\mathcal{V}\mathcal{H}\mathcal{H}}(U,X)Y-T_{V}A_{X}Y+A_{X}T_{U}Y.\\
    \end{eqnarray*} 
   \begin{eqnarray*}
   \mathcal{V}R(X,Y)U &=& R^{\mathcal{H}\mathcal{H}\mathcal{V}}(X,Y)U+A_{Y}A_{X}U-A_{X}A_{Y}U.\\  
     \mathcal{H}R(X,Y)U &=& \mathcal{H}(\nabla_{Y}A)_{X}U -\mathcal{H}(\nabla_{X}A)_{Y}U+T_{A_{X}Y}U-T_{A_{Y}X}U\\&&-T_{Tor(\nabla)(X,Y)}U-A_{Tor(\nabla)(X,Y)}U.\\
   \mathcal{V}R(X,Y)Z &=& \mathcal{V}(\nabla_{Y}A)_{X}Z -\mathcal{V}(\nabla_{X}A)_{Y}Z+T_{A_{X}Y}Z-T_{A_{Y}X}Z\\&&-T_{Tor(\nabla)(X,Y)}Z-A_{Tor(\nabla)(X,Y)}Z.\\   
    \mathcal{H}R(X,Y)Z &=& R^{\mathcal{H}\mathcal{H}\mathcal{H}}(X,Y)Z+A_{Y}A_{X}Z-A_{X}A_{Y}Z.
\end{eqnarray*} 
 \end{theorem}

Now, for semi-Riemannian manifolds $(\mathbf{M},g_{m})$, $(\mathbf{B},g_{b})$ with affine connections $\nabla$ and $\nabla^*$ and the dual connections $\overline{\nabla}$ and $\overline{\nabla^{*}}$ respectively we prove

\begin{proposition}
Let $\pi : (\mathbf{M},g_{m}) \longrightarrow (\mathbf{B},g_{b})$ be a conformal submersion. Then, $\pi: (\mathbf{M},\nabla) \longrightarrow (\mathbf{B}, \nabla^{*}) $ is a conformal submersion with horizontal distribution  if and only if  $\pi: (\mathbf{M},\overline{\nabla}) \longrightarrow (\mathbf{B}, \overline{\nabla^{*}}) $ is a conformal submersion with horizontal distribution. 
\end{proposition}

\begin{proof}
Consider, 
\begin{small}
\begin{eqnarray}
 \tilde{X}g_{m}(\tilde{Y},\tilde{Z}) &=& 2e^{2\phi}\tilde{X}(\phi)g_{b}(Y,Z)+e^{2\phi}Xg_{b}(Y,Z) \nonumber\\ 
 &=& 2e^{2\phi}\tilde{X}(\phi)g_{b}(Y,Z)+ e^{2\phi}\{g_{b}(\nabla^{*}_{X}Y,Z)+ g_{b}(Y,\overline{\nabla^{*}}_{X}Z)\}. \label{cf1}
 \end{eqnarray}
 \end{small}
 Now consider
 \begin{eqnarray}\label{cs4}
 \tilde{X}g_{m}(\tilde{Y},\tilde{Z}) &=& g_{m}(\nabla_{\tilde{X}}\tilde{Y}, \tilde{Z}) + g_{m}(\tilde{Y}, \overline{\nabla}_{\tilde{X}}\tilde{Z}) \nonumber\\
 &=& e^{2\phi}g_{b}(\pi_{*}(\nabla_{\tilde{X}}\tilde{Y}),Z)+ e^{2\phi}g_{b}(Y,\pi_{*}(\overline{\nabla}_{\tilde{X}}\tilde{Z})).
 \end{eqnarray}
 Since, 
  \begin{eqnarray}\label{cs5}
\pi_{*}(\nabla_{\tilde{X}}\tilde{Y}) &=& \nabla^{*}_{X}Y +\tilde{X}(\phi)Y+ \tilde{Y}(\phi)X - e^{2\phi}\pi_{*}(grad_{\pi}\phi)g_{b}(X,Y).
\end{eqnarray}
from the equations (\ref{cf1}),(\ref{cs4}) and (\ref{cs5}) we get 
\begin{eqnarray*}
\pi_{*}(\overline{\nabla}_{\tilde{X}}\tilde{Z}) &=& \overline{\nabla^{*}}_{X}Z +\tilde{X}(\phi)Z+ \tilde{Z}(\phi)X - \pi_{*}(grad_{\pi}\phi)e^{2\phi}g_{b}(X,Z).
\end{eqnarray*}
Hence,  $\pi: (\mathbf{M},\overline{\nabla}) \longrightarrow (\mathbf{B}, \overline{\nabla^{*}}) $ is a conformal submersion with horizontal distribution.\\ Converse is obtained by interchanging  $\nabla$, $\nabla^{*}$  with $\overline{\nabla}$, $\overline{\nabla^{*}}$ in the above proof.
\end{proof}

\section{Geodesics}
In this section, for a conformal submersion with horizontal distribution we prove a necessary and sufficient condition for $\pi\circ \sigma $ to be a geodesic of $\mathbf{B}$ when $\sigma$ is a geodesic of $\mathbf{M}$.
Let $M$, $B$ be Riemannian manifolds and $\pi :M\rightarrow B$ be a submersion. Let $E$ be a vector field on a curve $\sigma$ in $\mathbf{M}$ and the horizontal part $\mathcal{H}(E)$ and the vertical part $\mathcal{V}(E)$ of $E$ be denoted by $H$ and $V$, respectively. $\pi \circ \sigma$ is a curve in $B$ and $E_{*}$ denote the vector field $\pi_{*}(E) = \pi_{*}(H)$ on the curve $\pi \circ \sigma$ in $\mathbf{B}$. $E_{*}^{'}$ denote the covariant derivative of $E_{*}$ and is a vector field on $\pi \circ \sigma$. The horizontal lift to $\sigma$ of $E_{*}^{'}$ is denoted by $\tilde{E_{*}^{'}}$. In \cite{o1967submersions}, O'Neil compared the geodesics for semi-Riemannian submersion and Abe and Hasegawa \cite{abe2001affine} have done it for affine submersion with horizontal distribution.

Let $\pi :( \mathbf{M}, \nabla, g_{m} )\rightarrow(\mathbf{B},\nabla^{*}, g_{b})$ be a conformal submersion with horizontal distribution $\mathcal{H}(\mathbf{M})$. Throughout this section we assume $\nabla$ is torsion free. A curve $\sigma$ is a geodesic if and only if $\mathcal{H}(\sigma ^{''})=0$ and $\mathcal{V}(\sigma ^{''})=0$, where $\sigma^{''}$ is the covariant derivative of $\sigma^{'}$. So, first we obtain the equations for $\mathcal{H}(E^{'})$ and $\mathcal{V}(E^{'})$ for a vector field $E$ on the curve $\sigma$ in $\mathbf{M}$ for a conformal submersion with horizontal distribution.  

\begin{theorem}
Let $ \pi: (\mathbf{M},\nabla,g_{m}) \rightarrow (\mathbf{B},\nabla^{*},g_{b})$ be a conformal submersion with horizontal distribution, and let $E = H + V$ be a vector field on curve $\sigma$ in $\mathbf{M}$. Then, we have
\begin{eqnarray*}
\pi_{*}(\mathcal{H}(E^{'})) &=& E_{*}^{'}+\pi_{*}(A_{X}U+A_{X}V+T_{U}V)-e^{2\phi}\pi_{*}(grad_{\pi}\phi)g_{b}(\pi_{*}X,\pi_{*}H)\\&&+X(\phi)\pi_{*}H+H(\phi)\pi_{*}X,\\
\mathcal{V}(E^{'}) &=& A_{X}H+T_{U}H+\mathcal{V}(V^{'}),
\end{eqnarray*}
where $X= \mathcal{H}(\sigma^{'})$ and $U= \mathcal{V}(\sigma^{'})$.
\end{theorem}

\begin{proof}
 Consider a neighborhood of an arbitrary point $\sigma(t)$ of the curve $\sigma$ in $\mathbf{M}$. By choosing the base fields $W_{1},....,W_{n}$, where $n= dim \mathbf{B}$, near $\pi(\sigma(t))$ on $\mathbf{B}$ and an appropriate vertical base field near $\sigma(t)$, we can derive

\begin{eqnarray}
(E_{*}^{'})_{t} &=& \sum_{i}r^{i'}(t)(W_{i})_{\pi(\sigma(t))} + \sum_{i,k} r^{i}(t)s^k(t)(\nabla^{*}_{W_{k}}W_{i})_{\pi(\sigma(t))},\label{cg3}
\end{eqnarray}   
\begin{eqnarray}\label{cg4}
\pi_{*}(\mathcal{H}(E^{'})_{t}) &=& \sum_{i} r^{i'}(t)(W_{i})_{\pi(\sigma(t))}+ \sum_{i,k}r^{i}(t)s^{k}(t)\pi_{*}(\mathcal{H}(\nabla_{\tilde{W_{k}}}\tilde{W_{i}}))_{\pi(\sigma(t))} \nonumber\\ 
&&+ \pi_{*}((A_{H}U)+(A_{X}V)+(T_{U}V))_{\pi(\sigma(t))},
\end{eqnarray}
where $\tilde{W}_{i}$ be the horizontal lift of $W_{i}$, for $i=1,2,...n$ and $r^{i}(t)$ ( respectively $s^{k}(t)$) be the coefficients of $H$ (respectively of $X$ ) in the representation using the base fields  $\tilde{W}_{i}$ restricted to $\sigma$.

Since $\pi$ is a conformal submersion with horizontal distribution
\begin{small}
\begin{eqnarray*}
\pi_{*}(\mathcal{H}(\nabla_{H}X)) = \nabla^{*}_{\pi_{*}(H)}\pi_{*}X + X(\phi)\pi_{*}H+ H(\phi)\pi_{*}X- e^{2\phi}\pi_{*}(grad_{\pi}\phi)g_{b}(\pi_{*}X,\pi_{*}H).
\end{eqnarray*}
\end{small}
Hence
\begin{small}
\begin{eqnarray*}
\pi_{*}(\mathcal{H}(E^{'})) &=& E_{*}^{'}+\pi_{*}(A_{X}U+A_{X}V+T_{U}V)-e^{2\phi}\pi_{*}(grad_{\pi}\phi)g_{b}(\pi_{*}X,\pi_{*}H)\\&& +X(\phi)\pi_{*}H+H(\phi)\pi_{*}X.
\end{eqnarray*}
\end{small}
Similarly we can prove $\mathcal{V}(E^{'}) = A_{X}H+T_{U}H+\mathcal{V}(V^{'}).$
\end{proof}
\vspace{.2cm}

 For $\sigma^{''}$ we have

\begin{corollary}
Let $\sigma$ be a curve in $\mathbf{M}$ with $X = \mathcal{H}(\sigma^{'})$ and $U = \mathcal{V}(\sigma^{'})$. Then, 
\begin{eqnarray}
\pi_{*}(\mathcal{H}(\sigma^{''})) &=& \sigma_{*}^{''}+\pi_{*}(2A_{X}U+T_{U}U)-e^{2\phi}\pi_{*}(grad_{\pi}\phi)g_{b}(\pi_{*}X,\pi_{*}X)\nonumber\\&&+2X(\phi)\pi_{*}X, \label{cg5}\\
\mathcal{V}(\sigma^{''}) &=& A_{X}X+T_{U}X+\mathcal{V}(U^{'}), \label{cg6}
\end{eqnarray}
where $\sigma_{*}^{''}$ denotes the covariant derivative of $(\pi \circ \sigma)^{'}$.
\end{corollary}

Now for a conformal submersion with horizontal distribution we prove a necessary and sufficient condition for $\pi \circ \sigma$ to become a geodesic of $\mathbf{B}$ when $\sigma$ is a geodesic of $\mathbf{M}$.

 \begin{theorem}
 Let $ \pi: (\mathbf{M},\nabla,g_{m}) \rightarrow (\mathbf{B},\nabla^{*},g_{b})$ be a conformal submersion with horizontal distribution. If $\sigma$ is a geodesic of $\mathbf{M}$, then $\pi \circ \sigma$ is a geodesic of $\mathbf{B}$ if and only if 
 
 \begin{eqnarray*}
 \pi_{*}(2A_{X}U+T_{U}U))+2X(\phi)\pi_{*}X =\pi_{*}(grad_{\pi}\phi) \parallel X \parallel^{2},
 \end{eqnarray*}
 where $X = \mathcal{H}(\sigma^{'})$ and $U = \mathcal{V}(\sigma^{'})$ and $\parallel X \parallel^{2} = g_{m}(X,X)$.
 \end{theorem}
 
\begin{proof}
Since $\sigma$ is a geodesic on $\mathbf{M}$ from the equation (\ref{cg5}) 

\begin{eqnarray*}
 \sigma_{*}^{''} &=& \pi_{*}(grad_{\pi}\phi) \parallel X \parallel^{2} - \pi_{*}(2A_{X}U+T_{U}U) - 2d\phi(X)\pi_{*}X.
\end{eqnarray*}
Hence, $\pi \circ \sigma$ is a geodesic on $\mathbf{B}$ if and only if 
\begin{eqnarray*}
\pi_{*}(2A_{X}U+T_{U}U))+2X(\phi)\pi_{*}X =\pi_{*}(grad_{\pi}\phi) \parallel X \parallel^{2}.
 \end{eqnarray*}  
  
\end{proof}
\begin{remark}
If $\sigma$ is a horizontal geodesic (that is, $\sigma$ is a geodesic with $\mathcal{V}(\sigma^{'})=0$), then $\pi \circ \sigma$ is a geodesic if and only if $2X(\phi)\pi_{*}X =\pi_{*}(grad_{\pi}\phi) \parallel X \parallel^{2}$.
\end{remark}

\begin{definition}
Let $ \pi: (\mathbf{M},\nabla,g_{m}) \rightarrow (\mathbf{B},\nabla^{*},g_{b})$ be a conformal submersion with horizontal distribution and $\alpha$ be a smooth curve in $\mathtt{B}$. Let $\alpha^{'}$ be the tangent vector field of $\alpha$ and $\tilde{(\alpha^{'})}$ be its horizontal lift. Define the horizontal lift of $\alpha$ as the integral curve $\sigma$ on $\mathbf{M}$ of $(\tilde{\alpha^{'})}$.
\end{definition}
Now, we have 
\begin{proposition}\label{cg10}
Let $ \pi: (\mathbf{M},\nabla,g_{m}) \rightarrow (\mathbf{B},\nabla^{*},g_{b})$ be a conformal submersion with horizontal distribution  such that $ A_{Z}Z =0$ for all horizontal vector fields $Z$. Then, every horizontal lift of a geodesic of $\mathbf{B}$ is a geodesic of $\mathbf{M}$ if and only if $2X(\phi)\pi_{*}X =\pi_{*}(grad_{\pi}\phi) \parallel X \parallel^{2}$, where $X$ is the horizontal part of the tangent vector field of the horizontal lift of the geodesic on $\mathbf{B}$.
\end{proposition}

\begin{proof}
Let $\alpha$ be a geodesic on $\mathbf{B}$, $\sigma$ be the horizontal lift of $\alpha$. Then, we have $\pi \circ \sigma = \alpha$ and $\sigma^{'}(t)$ = $\tilde{(\alpha^{'}(t))}$. Let $X = \mathcal{H}(\sigma^{'}(t))$ and $U = \mathcal{V}(\sigma^{'}(t))$, clearly $X = \tilde{(\alpha^{'}(t))}$ and $U = 0$. Then, from the equations (\ref{cg5}) and (\ref{cg6}) 
\begin{eqnarray*}
\pi_{*}(\mathcal{H}(\sigma^{''})) &=& \alpha^{''}-\pi_{*}(grad_{\pi}\phi) \parallel X \parallel^{2}+2X(\phi)\pi_{*}X. \\
\mathcal{V}(\sigma^{''}) &=& A_{X}X. 
\end{eqnarray*}
Since $\alpha$ is a geodesic and $ A_{X}X = 0$ we have, $\sigma^{''} = 0$ if and only if  $2X(\phi)\pi_{*}X =\pi_{*}(grad_{\pi}\phi) \parallel X \parallel^{2}$. That is, every horizontal lift of a geodesic of $\mathbf{B}$ is a geodesic of $\mathbf{M}$ if and only if $2X(\phi)\pi_{*}X =\pi_{*}(grad_{\pi}\phi) \parallel X \parallel^{2}$.
\end{proof}
 \begin{corollary} 
 Let $ \pi: (\mathbf{M},\nabla,g_{m}) \rightarrow (\mathbf{B},\nabla^{*},g_{b})$ be a conformal submersion with horizontal distribution  such that $ A_{Z}Z =0$ for all horizontal vector fields $Z$. Then, $\nabla^{*}$ is geodesically complete if $\nabla$ is geodesically complete and $2X(\phi)\pi_{*}X =\pi_{*}(grad_{\pi}\phi) \parallel X \parallel^{2}$, where $X$ is the horizontal part of the tangent vector field of the horizontal lift of the geodesic on $\mathbf{B}$.
\end{corollary}

\begin{proof}
Let $\alpha$ be a geodesic of $\mathbf{B}$ and $\tilde{\alpha}$ be its horizontal lift to $\mathbf{M}$, by Proposition (\ref{cg10})  $\tilde{\alpha}$ is a geodesic on $\mathbf{M}$. Since $\nabla$ is geodesically complete, $\tilde{\alpha}$ can be defined on the entire real line. Then, the projected curve of the extension of $\tilde{\alpha}$ is a geodesic and is the extension of $\alpha$, that is $\nabla^{*}$ is geodesically complete.
\end{proof}

\section*{Acknowledgement}
 Mahesh T V thanks Indian Institute of Space Science and Technology (IIST), Department of Space, Government of India for the award of Doctoral Research Fellowship.

%

\bibliographystyle{spmpsci}      


\end{document}